\documentclass[12pt]{article} 
\usepackage{graphicx}
\usepackage{amssymb}
\usepackage{mathtools}
\usepackage{amsmath}
\usepackage{amsthm}
\usepackage{color}
\usepackage[hidelinks]{hyperref}
\usepackage{csquotes}
\usepackage[pagewise]{lineno}
\usepackage[english]{babel}
\usepackage{enumitem}
\usepackage{graphicx,url} 
\usepackage{tocloft}
\usepackage{sectsty}
\usepackage{mathrsfs}
\usepackage{eufrak}
\usepackage{comment}
\usepackage[title]{appendix}
\usepackage{multirow}
\usepackage{booktabs}

\usepackage[T1]{fontenc}

\makeatletter
\renewcommand\tableofcontents{%
	\null\hfill\textbf{\Large\contentsname}\hfill\null\par
	\@mkboth{\MakeUppercase\contentsname}{\MakeUppercase\contentsname}%
	\@starttoc{toc}%
}
\makeatother

\numberwithin{equation}{section}

\sectionfont{\fontsize{13}{15}\selectfont}
\subsectionfont{\fontsize{12}{15}\selectfont}

\renewcommand{\Re}{\operatorname{Re}}
\newtheorem{thm}{Theorem}[section]

\newtheorem{lem}[thm]{Lemma}
\newtheorem{prop}[thm]{Proposition}
\newtheorem{defn}{Definition}[section]
\newtheorem{exmp}{Example}[section]
\newtheorem{rmk}{Remark}[section]

\numberwithin{equation}{section}

\renewcommand{\Re}{\operatorname{Re}}

\def \R{{\mathbb{R}}}
\def \P {\Phi(x)}
\def \o {\omega(x)}

\def \japxik{\langle \xi \rangle_k}
\def \japx{\langle x \rangle}
\def \hyp{Z_{ext}(N)}
\def \pd{Z_{int}(N)}

\def \J{[0,T] \times \R^{2n}}
\def \la{\langle}
\def \ra{\rangle}

\def \rak{\rangle_k}
\def \h{h(x,\xi)}

\def \wmt{(\omega,g_{\Phi,k}^{\rho,\tilde\rho})}

\providecommand{\keywords}[1]
{
	\small	
	\textbf{\text{Keywords:}} #1
}

\providecommand{\subclass}[1]
{
	\small	
	\textbf{\text{MSC(2010):}} #1
}

\usepackage{geometry}
\title{\Large Global Well-posedness of a Class of \\Singular Hyperbolic Cauchy Problems\thanks{Dedicated to Bhagawan Sri Sathya Sai Baba on the Occassion of His 96th Birthday.}}

\author{\normalsize Rahul Raju Pattar\thanks{rahulrajupattar@gmail.com (Corresponding Author)} , N. Uday Kiran\thanks{nudaykiran@sssihl.edu.in}  \\
	\small Department of Mathematics and Computer Science\\
	\small Sri Sathya Sai Institute of Higher Learning, Puttaparthi, India \\
}

\date{}

\begin{document} 
	
	\maketitle	
	
	\begin{abstract}
		The goal of this paper is to establish a global well-posedness, cone condition and loss of regularity for singular hyperbolic equations with coefficients in { $L^1((0,T];C^\infty(\mathbb{R}^n)) \cap C^1((0,T];C^\infty(\mathbb{R}^n))$} and Cauchy data in an appropriate Sobolev space tailored to a metric on the phase space.
		The coefficients are unbounded near the singular hyperplane $t=0$ and polynomially growing as $|x| \to \infty.$ The singular behavior is characterized by the blow-up rate of the coefficients and their first $t$-derivatives near $t=0.$
		In order to study the interplay of the singularity in $t$ and unboundedness in $x$, we consider a class of metrics on the phase space. Our methodology relies on the use of the Planck function associated to the metric to subdivide the extended phase and to define an infinite order pseudodifferential operator for the conjugation.
		We also give some counterexamples.\\
		\\
		\subclass{35L81 $\cdot$ 35L15 $\cdot$ 35S05 $\cdot$ 35B65 $\cdot$ 35B30}
	\end{abstract}
	
	
	\section{Introduction}
	We consider the singular hyperbolic Cauchy problems of the form
	\[
	u_{tt} + A(t,x,D_x)u_t + B(t,x,D_x)u = f, \quad (t,x)  \in (0,T] \times \R^n, \; T<\infty,
	\]
	where $u$ is a function of $t,$ valued in a Sobolev space while $A$ and $B$ are linear partial differential operators some of whose coefficients tend to infinity in some sense as $t \to 0.$ See \cite{CS} and the references therein for a discussion on singular hyperbolic Cauchy problems.
	A study of such problems is motivated by applications in physics \cite{GY}, in the study of blow-up solutions of quasilinear problems \cite[Section 4]{hiro} and also in the study certain degenerate Cauchy problems \cite{CS}.
	One of the prominent features of singular hyperbolic Cauchy problems is the loss of regularity index of the solutions in relation to the initial data defined in an appropriate Sobolev space. 
	
	In order to describe the loss of regularity and well-posedness results from the literature, consider the following model strictly hyperbolic equation
	\begin{equation}\label{mod}
		\begin{cases}
			\begin{aligned}
				&\partial_t^2u - a(t,x)\partial_x^2u + \sum_{\substack{j,l=0 \\ j+l \leq 1}}^{1} b_{j,l}(t,x) \partial_x^j \partial_t^lu = f(t,x), \quad (t,x) \in (0,T] \times \R, \\
				&u(0,x) = f_1(x), \; \partial_t(0,x) = f_2(x).
			\end{aligned}
		\end{cases}
	\end{equation}
	The operator coefficients are in $L^1\left((0,T]; C^\infty(\R)\right) \cap C^1\left((0,T]; C^\infty(\R)\right)$ and the singular behavior of the above Cauchy problem is described by the following estimates
	\begin{equation}\label{sing}
		\begin{rcases}
			\begin{aligned}
				|\partial_x^\beta \partial_t a(t,x)| &\leq C_{\beta}^{(1)} \; \o^2\P^{-|\beta|} \; \frac{1}{t^q} |\ln t|^{(\tilde\gamma-1){\bf I}_q},\\
				|\partial_x^\beta a(t,x)| &\leq C_{\beta}^{(2)} \; \o^2\P^{-|\beta|} \; \frac{1}{t^p} |\ln t|^{\tilde\gamma {\bf I}_q},\\
				|\partial_x^\beta b_{j,l}(t,x)| &\leq C_{\beta}^{(3)} \; \o^j \P^{-|\beta|} \; \frac{1}{t^r},
			\end{aligned}
		\end{rcases}
	\end{equation}
	with $\beta \in \mathbb{N}_0, C_{\beta}^{(i)}>0,i=1,2,3,$ $q \in [1,\infty), p \in [0,1), p \leq q-1$ and $ r \in [0,1).$ The function ${\bf I}_q$ is such that ${\bf I}_q \equiv 1$ if $q=1$ else ${\bf I}_q \equiv 0.$ The functions $\o$ and $\P$ are positive monotone increasing in $|x|$ such that $1\leq \o \lesssim \P\lesssim \japx = (1+|x|^2)^{1/2}.$ They specify the structure of the differential equation in the space variable. These functions will be discussed in detail in Section \ref{tools}. 
	
	In this work, our main interest is on the optimality of loss when the coefficients are singular in time and unbounded in space. In particular, our interest is either  blow-up or infinitely many oscillations near $t=0$ and polynomial growth in $x$. 
	Such equations have a well-known behavior of a loss of derivatives when the coefficients are bounded in space together with all their derivatives (see for example \cite{CSK,Cico1}). Along with the extension of the results to a global setting ($x \in \R^n$ and the coefficients are allowed to grow polynomially in $x$), we also investigate the behavior at infinity of the solution in relation to the coefficients. By a loss of regularity index of the solution in relation to the initial datum we mean a change of indices in an appropriate Sobolev space.
	See Table \ref{T1} for the well-posedness results from the literature in the context of singular hyperbolic Cauchy problems.
	
	\begin{table}[h]
		\begin{tabular}{@{}|cccc|c|cc|c|c|@{}}
			\toprule
			
			\multicolumn{4}{|c|}{\begin{tabular}[c]{@{}c@{}}\footnotesize Order of Singularity\\ \footnotesize at $t=0$\end{tabular}}                                                                                               & \multirow{2}{*}{\begin{tabular}[c]{@{}c@{}}\footnotesize Regularity \\ \footnotesize  in $t$ of \\ \footnotesize coefficients\end{tabular}} & \multicolumn{2}{c|}{\begin{tabular}[c]{@{}c@{}}\footnotesize Growth in $x$\\ \footnotesize of coefficients\end{tabular}} & \multirow{2}{*}{\begin{tabular}[c]{@{}c@{}}\footnotesize Loss of Regularity\\ \footnotesize index for Solution\end{tabular}} & \multirow{2}{*}{\footnotesize Ref.}             \\ \cmidrule(r){1-4} \cmidrule(lr){6-7}       
			  
			\multicolumn{1}{|c|}{$p$}                                 & \multicolumn{1}{c|}{$q$}                                      & \multicolumn{1}{c|}{$\tilde\gamma$} & $r$              &                                    & \multicolumn{1}{c|}{$\omega$}             & $\Phi$             &                                                        &                                  \\ \midrule
			\multicolumn{1}{|c|}{$0$}                                 & \multicolumn{1}{c|}{$1$}                                      & \multicolumn{1}{c|}{$(0,1)$}        & 0                & \footnotesize$C^1((0,1])$                       & \multicolumn{1}{c|}{$\omega(x)$}          & $\Phi(x)$          & \multicolumn{1}{c|}{\begin{tabular}[c]{@{}c@{}} \footnotesize arbitrarily\\ \footnotesize small\end{tabular}}            & \footnotesize \cite{RahulNUK5} \\ \midrule
			\multicolumn{1}{|c|}{$0$}                                 & \multicolumn{1}{c|}{$1$}                                      & \multicolumn{1}{c|}{$1$}            & 0                & \footnotesize$C^1((0,1])$                       & \multicolumn{1}{c|}{1}                    & 1                  & \multicolumn{1}{c|}{\footnotesize Finite}                                                                            & \footnotesize \cite{Cico1}     \\ \midrule
			\multicolumn{1}{|c|}{$0$}                                 & \multicolumn{1}{c|}{$1$}                                      & \multicolumn{1}{c|}{$1$}            & 0                & \footnotesize$C^1((0,1])$                       & \multicolumn{1}{c|}{$\omega(x)$}          & $\Phi(x)$          & \multicolumn{1}{c|}{\footnotesize Finite}                                                             & \ \footnotesize\cite{RahulNUK3} \\ \midrule
			\multicolumn{1}{|c|}{$0$}                                 & \multicolumn{1}{c|}{$1$}                                      & \multicolumn{1}{c|}{$[1,\infty)$}   & 0                & \footnotesize$C^2((0,T])$                       & \multicolumn{1}{c|}{$\omega(x)$}          & $\Phi(x)$          & \multicolumn{1}{c|}{\begin{tabular}[c]{@{}c@{}}\footnotesize Ranging from\\ \footnotesize Finite to Infinite\end{tabular}} &  \footnotesize\cite{RahulNUK4} \\ \midrule
			\multicolumn{1}{|c|}{$[0,1)$}                             & \multicolumn{1}{c|}{$(1, \infty)$}                            & \multicolumn{1}{c|}{-}              & 0                & \footnotesize$C^1((0,1])$                       & \multicolumn{1}{c|}{1}                    & 1                  & \multicolumn{1}{c|}{\footnotesize Infinite}                                                                              & \footnotesize\cite{CSK}       \\ \midrule
			\multicolumn{1}{|c|}{\boldmath{$\Big[0, \frac{1}{2}\Big)$}} & \multicolumn{1}{c|}{\boldmath{$\left( 1, \frac{3}{2} \right)$}} & \multicolumn{1}{c|}{\textbf{-}}     & \boldmath{$[0,1)$} & \footnotesize \boldmath{$C^1((0,1])$}              & \multicolumn{1}{c|}{\boldmath{$\omega(x)$}} & \boldmath{$\Phi(x)$} & \multicolumn{1}{c|}{\footnotesize \textbf{Infinite }}                                                         & \footnotesize\textbf{**}                 \\ \bottomrule
		\end{tabular}
		\label{T1}
		\caption{Well-posedness results for singular hypebolic Cauchy problems. (**) refers to Theorem \ref{result1}  of this paper. }
	\end{table}
	
	In \cite{Cico1}, Cicognani discussed the well-posedness of (\ref{mod}) for the case $\o=\P=1$ and $p=0,q=1, \tilde\gamma=1,r=0$ in (\ref{sing}), where the author reports well-posedness in $C^\infty(\R)$ for 
	the Cauchy problem (\ref{mod}) with a finite loss of derivatives.
	Colombini et al. \cite{CSK} considered the Cauchy problem (\ref{mod}) with operator coefficients independent of $x$ and singular behavior prescribed by the parameters $p \in [0,1), q \in (1, \infty)$ and $ r=0$. They report well-posedness in Gevrey space $G^s, 1 \leq s < \frac{q-p}{q-1},$ with infinite loss of derivatives. We study in \cite{Rahul_NUK2} the case of $p=0,q \in \left( 1, \frac{3}{2}\right)$ and $r=0$ with generic structure functions $\omega$ and $\Phi$ in (\ref{sing}) and report infinite loss of both derivatives and decay.
	
	In this paper our interest is in the operator coefficients that are $L^1$ integrable in $t$ but singular near $t=0$  and are of polynomial growth in $x$. In particular, our interest is $p \in \big[0,\frac{1}{2}), q \in \left( 1, \frac{3}{2}\right), p \leq q-1 , r \in [0,1)$ and polynomial growth in $x$ prescribed by $\o, \P$ in (\ref{sing}). 
	An example of a coefficient $a(t,x)$ satisfying (\ref{sing})  is given below.
	
	\begin{exmp}
		Let $T=1$, $\kappa_1 \in[0,1]$ and $\kappa_2\in(0,1]$ such that $\kappa_1 \leq \kappa_2$. Then, 
		\begin{linenomath*}
			$$
			a(t,x) = \japx^{2\kappa_1}  \left( 2+\cos  \japx^{1-\kappa_2} \right) \left( \frac{1}{t^{1/4}} \left( 2+ \sin\left( \frac{1}{t^{1/8}}\right) \right) \right)
			$$
		\end{linenomath*}
		satisfies the estimates (\ref{sing}) for $\o = \japx^{\kappa_1}$, $\P = \japx^{\kappa_2},$ $ p= \frac{1}{4},$ $ q= \frac{11}{8}$. The example shows that singular coefficients can also have infinitely many oscillations near $t=0.$
	\end{exmp}
	
	In order to study the interplay between the singularity in time and unboundedness in space one needs to consider an appropriate metric on the phase space \cite{Lern,nicRodi}. 
	In our case, we study the Cauchy problem (\ref{mod}) by considering the metrics
	\begin{equation}\label{m1}
		g_{\Phi,k} = \P^{-2} \vert dx\vert^2 + \japxik^{-2}\vert d\xi\vert^2,
	\end{equation}  
	where $ \japxik = (k^2+|\xi|^2)^{1/2},$ for sufficiently large $k$ chosen appropriately. These metrics are discussed in Section \ref{metric}.
	From the estimates (\ref{sing}), note that $\omega$ and $\Phi$ are associated with the weight and metric respectively, and they specify the structure of the differential equation in the space variable. These functions will be discussed in detail in Section \ref{structfns}.   
	We report that the solution not only experiences a loss of derivatives but also a decay in relation to the initial datum defined in a Sobolev space modelled by the infinite order pseudodifferential operator 
	\begin{equation}\label{iopdo}
		e^{\Lambda(t)\Theta(x,D_x)}.
	\end{equation}
	Here $\Lambda \in C([0,T])$ and the symbol of the operator $\Theta(x,D_x)$ is given by $h(x, \xi)^{-1/\sigma}=$ $(\Phi(x)\langle\xi\rangle_{k})^{1/\sigma}$ where $h(x, \xi)$ is the Planck function related to the metric $g_{\Phi, k}$ in (\ref{m1}) and $3 \leq \sigma < (q-p)/(q-1)$. The operator $\Theta(x,D_x)$ explains the quantity of the loss by linking it to the metric on the phase space and the singular behavior while $\Lambda(t)$ gives a scale for the loss. Hence, we call the conjugating operator as loss operator.
	
	Our methodology relies upon two important techniques: the subdivision of the extended phase space into two regions and conjugation of a first order system corresponding to the operator $P$ in (\ref{eq1}) by the loss operator. Both these techniques result in the change of the metric governing the operator where the new metric is conformally equivalent to the one in (\ref{m1}). As seen in (\ref{root1}), the characteristic roots corresponding to the operator $P$ showcase a stronger singular behavior compared to the  principal symbol. Due to the subdivision of the phase space, this results in the change of the metric as demonstrated in Lemma \ref{lemroot}. This metric is of the form
	\begin{equation}\label{me2}
		\tilde{g}_{\Phi, k}^{(1)}=(\P\japxik)^{2\delta'} {g}_{\Phi, k},
	\end{equation}
	where $\delta'= \frac{p}{q-p}.$ On the other hand, our work in \cite[Theorem 4.0.1]{Rahul_NUK2} suggests that the conjugation by the loss operator changes the metric to
	\begin{equation}\label{me3}
		\tilde{g}_{\Phi, k}^{(2)} = (\P\japxik)^{2/\sigma} {g}_{\Phi, k}.
	\end{equation}
	As our approach to establish well-posedness is based on the energy estimates, we consider the metric $\tilde{g}_{\Phi, k}^{(2)}  = \tilde{g}_{\Phi, k}^{(1)} \vee \tilde{g}_{\Phi, k}^{(2)} $ for the application of the sharp G\r{a}rding inequality \cite[Theorem 18.6.14]{Horm}.
	
	From the energy estimate used in proving the well-posedness we derive an optimal cone condition for the solution of the Cauchy problem (\ref{eq1}) in Section \ref{cone}.  Though the characteristics of the operator $P$ in (\ref{eq1}) are singular in nature, the $L^1$ integrability of the singularity guarantees that the propagation speed is finite. The weight function governing the coefficients influences the geometry of the slope of the cone due to which the cone condition in our case is anisotropic in nature.
	
	In Section \ref{CE}, we give a set counterexamples showing how the lower order terms influence the well-posedness and the loss of regularity when $L^1$ integrability condition is violated. We cover various cases such as no loss, finite loss and nonuniqueness.

	The paper is organized as follows. In Section \ref{tools}, we describe the tools necessary for our analysis.
	In Section \ref{stmt}, we define a Cauchy problem of our interest and state the well-posedness result whose proof will be presented in Section \ref{Proof1}. 
	In Section \ref{Symbol classes}, we define appropriate generalized parameter dependent symbol classes.
	In Section \ref{cone} we derive a cone condition, while in Section \ref{CE} we provide the set of examples.

	\section{Tools}\label{tools}
	
	In this section we introduce the main tools of this paper. The first part is devoted to a class of metrics on the phase space that govern the geometry of the symbols in our consideration while the second for properties of structure functions. In the third part, we define Sobolev spaces associated to the metric.  Lastly, we device a localization technique based on the Planck function associated to the metric.
	
	\subsection{Our Choice of Metric on the Phase Space}\label{metric}
	
	In this section, we review some notation and terminology used in the study of metrics on the phase space, see \cite[Chapter 2]{Lern} and \cite{nicRodi} for details. Let us denote by $\sigma(X,Y)$ the standard symplectic form on $T^*\R^n\cong \R^{2n}$: if $X=(x,\xi)$ and $Y=(y,\eta)$, then $\sigma$ is given by
	\begin{linenomath*}
		\[
		\sigma(X,Y)=\xi \cdot y - \eta \cdot x.
		\]	
	\end{linenomath*}
	We can identify $\sigma$ with the isomorphism of $\R^{2n}$ to $\R^{2n}$ such that $\sigma^*=-\sigma$, with the formula $\sigma(X,Y)= \langle \sigma X,Y\rangle$. Consider a Riemannian metric $g_X$ on $\R^{2n}$ (which is a measurable function of $X$) to which we associate the dual metric $g_X^\sigma$ by
	\begin{linenomath*}
		\[ 
		g_X^\sigma(Y)= \sup_{0 \neq Y' \in \R^{2n}} \frac{\langle \sigma Y,Y'\rangle^2}{g_X(Y')}, \quad \text{ for all } Y \in \R^{2n}.
		\]
	\end{linenomath*}
	
	Considering $g_X$ as a matrix associated to positive definite quadratic form on $\R^{2n}$, $g_X^\sigma=\sigma^*g_X^{-1}\sigma$.
	We define the Planck function \cite{nicRodi} which plays a crucial role in the development of pseudodifferential calculus as
	\begin{linenomath*}
		\[ 
		h_g(x,\xi) := \sup_{0\neq Y \in \R^{2n}} \Bigg(\frac{g_X(Y)}{g_X^\sigma(Y)}\Bigg)^{1/2}.
		\]
	\end{linenomath*}
	The uncertainty principle is quantified as the upper bound $h_g(x,\xi)\leq 1$. In the following, we often make use of the strong uncertainty principle, that is, for some $\kappa>0$, we have
	\begin{linenomath*}
		\[
		h_g(x,\xi) \leq (1+|x|+|\xi|)^{-\kappa}, \quad (x,\xi)\in \R^{2n}.
		\]
	\end{linenomath*}
	In general, we use the metrics of the form
	\begin{equation}\label{m2}
		g_{\Phi,k}^{\rho,r} = \left( \frac{\japxik^{\rho_2}}{\P^{\tilde\rho_1}} \right)^2 |dx|^2 +  \left( \frac{\P^{\tilde\rho_2}}{\japxik^{\rho_1}} \right)^2 |d\xi|^2.
	\end{equation}
	Here $\rho=(\rho_1,\rho_2)$ , $\tilde\rho=(\tilde\rho_1,\tilde\rho_2)$ for $\rho_j,\tilde\rho_j \in [0,1]$, $j=1,2$ are such that $0 \leq \rho_2<\rho_1 \leq 1$ and $0 \leq \tilde\rho_2<\tilde\rho_1 \leq 1$. 
	The Planck function associated to the metric in (\ref{m2}) is  $\P^{\tilde\rho_2-\tilde\rho_1} \japxik^{\rho_2-\rho_1}.$
	
	\subsection{Properties of the Structure Functions $\o$ and $\P$}\label{structfns}
	The functions $\o$ and $\P$ are associated with weight and metric respectively. They specify the structure of the differential equation. As pseudodifferential calculus is the datum of the metric satisfying some local and global conditions. In our case, it amounts to the conditions on $\Phi$. The symplectic structure and the uncertainty principle also play a natural role in the constraints imposed on $\Phi$. So we consider $\Phi$ to be a monotone increasing function of $|x|$ satisfying the following conditions: 
	\begin{linenomath*}
		\begin{alignat*}{3}
			1 \; \leq & \quad \Phi(x) &&\lesssim  1+|x| && \quad \text{(sub-linear)} \\
			\vert x-y \vert \; \leq & \quad r\Phi(y) && \implies C^{-1}\Phi(y)\leq \Phi(x) \leq C \Phi(y)  && \quad \text{(slowly varying)} \\
			&\Phi(x+y) && \lesssim  \Phi(x)(1+|y|)^s && \quad \text{(temperate)}
		\end{alignat*}
	\end{linenomath*}
	for all $x,y\in\R^n$ and for some $r,s,C>0$.
	
	For the sake of calculations arising in the development of symbol calculus related to the metrics $g_{\Phi,k}$, we need to impose following additional conditions:
	\begin{linenomath*}
		\begin{alignat*}{3}
			|\Phi(x) - \Phi(y)| \leq & \Phi(x+y) && \leq \Phi(x) + \Phi(y),  && \quad  (\text{Subadditive})\\
			& |\partial_x^\beta \Phi(x)| && \lesssim \Phi(x) \japx^{-|\beta|}, \\
			&\Phi(ax) &&\leq a\P, \text{ if } a>1,\\
			& a\P &&\leq \Phi(ax), \text{ if } a \in [0,1],
		\end{alignat*}
	\end{linenomath*}
	where $\beta \in \mathbb{Z}_+^n$. It can be observed that the above conditions are quite natural in the context of symbol classes. In our work, we need even the weight function $\omega$ to satisfy the above stated properties of $\Phi$. In order to arrive at an energy estimate using the Sharp G\r{a}rding inequality (see Section \ref{energy} for details), we impose the following condition 
	\begin{linenomath*}
		\[
		\o \lesssim \P, \quad x \in \R^n.
		\] 
	\end{linenomath*}
	
	\subsection{Sobolev Spaces}
	We now introduce the Sobolev space related to the metric $g_{\Phi, k}$ that is suitable for our analysis.
	
	\begin{defn}\label{Sobo} 
		The Sobolev space $H_{\Phi, k}^{s, \varepsilon, \sigma}\left(\R^{n}\right)$ for $\sigma>2, \varepsilon \geq 0$ and $s=\left(s_{1}, s_{2}\right) \in \R^{2}$ is defined as
		$$
		H_{\Phi, k}^{s, \varepsilon, \sigma}\left(\R^{n}\right)=\left\{v \in L^{2}\left(\mathbb{R}^{n}\right): \Phi(x)^{s_{2}}\la D\rak^{s_{1}} \exp \{\varepsilon\left(\Phi(x) \la D_{x}\rak \right)^{1 / \sigma}\} v \in L^{2}\left(\mathbb{R}^{n}\right)\right\}
		$$
		equipped with the norm $\|v\|_{\Phi, k ; s, \varepsilon, \sigma}= \|\Phi(\cdot)^{s_{2}}\la D\rak^{s_{1}} \exp \{\varepsilon\left(\Phi(\cdot)\la D\rak \right)^{1 / \sigma}\} v \|_{L^{2}}$. The operator $\exp \{\varepsilon\left(\Phi(x) \la D_{x} \rak \right)^{1 / \sigma} \}$ is an infinite order pseudodifferential operator with the Fourier multiplier $\exp \{\varepsilon\left( \Phi(x)\japxik \right)^{1 / \sigma} \}$.
	\end{defn}
	
	The subscript $k$ in the notation $H_{\Phi, k}^{s, \varepsilon, \sigma}\left(\R^{n}\right)$ is related to the parameter in the operator $\la D \ra_k = (k^2 - \Delta_x)^{1/2}.$ 
	
	\subsection{Subdivision of the Phase Space}\label{zones}
	One of the main tools in our analysis is the division of the extended phase space $J=\J,$ where $T>0,$ into two regions using the Planck function $h(x,\xi)=(\P \japxik)^{-1}$ of the metric $g_{\Phi,k}$ in (\ref{m1}). We use these regions in the proof of Theorem \ref{result1} (see Section \ref{factr}) to handle the low regularity in $t$. To this end we define the time splitting point $t_{x,\xi}$, for a fixed $(x,\xi)$, as the solution to the equation
	\begin{linenomath*}
		\[
		t^{q-p}=N\h,
		\]
	\end{linenomath*}
	where $N$ is the positive constant chosen appropriately later.  Since $3 \leq \sigma < \frac{q-p}{q-1},$ we consider $\delta \in (0,1)$ such that
	\begin{equation}\label{delta}
		\frac{1}{\sigma} = \frac{q-1+\delta}{q-p}.
	\end{equation}
	Implying
	\begin{equation}\label{gamma}
		\gamma := 1- \frac{1}{\sigma} = \frac{1-\delta-p}{q-p}.
	\end{equation}
	Using $t_{x,\xi}$ and (\ref{gamma}) we define the interior region
	\begin{linenomath*}
		\begin{equation} \label{zone1}
			\begin{aligned}
				\pd &=\{(t,x,\xi)\in J : 0 \leq t \leq t_{x,\xi}, \; |x| + |\xi| > N\}\\
				&= \{(t,x,\xi)\in J : t^{1-\delta-p} \leq N^\gamma \h^{\gamma}, \; |x| + |\xi| >  N\}
			\end{aligned}
		\end{equation}
	\end{linenomath*}
	and the exterior region
	\begin{linenomath*}
		\begin{equation} \label{zone2}
			\begin{aligned}
				\hyp &=\{(t,x,\xi)\in J : t_{x,\xi} < t \leq T, \; |x| + |\xi| >  N\}\\
				&= \{(t,x,\xi)\in J : t^{1-\delta-p} > N^\gamma \h^{\gamma}, \; |x| + |\xi| >  N\}.
			\end{aligned}
		\end{equation}
	\end{linenomath*}	
	We use these regions to define the parameter dependent global symbol classes in Section \ref{Symbol classes}.

	\section{Statement of the Main Result}\label{stmt}
	We consider the Cauchy problem 
	\begin{linenomath*}
		\begin{equation}
			\begin{cases}
				\label{eq1}
				P(t,x,\partial_t,D_x)u(t,x)= f(t,x), \qquad D_x = -i\nabla_x,\;(t,x) \in (0,T] \times \R^n, \\
				u(0,x)=f_1(x), \quad \partial_tu(0,x)=f_2(x),
			\end{cases}
		\end{equation}
	\end{linenomath*}
	with the strictly hyperbolic operator $P(t,x,\partial_{t},D_{x}) = \partial_t^2 + b_0(t,x)\partial_t+ a(t,x,D_x)+ b(t,x,D_x)$ where
	\begin{linenomath*}			
		\[
		a(t,x,\xi)  = \sum_{i,j=1}^{n} a_{i,j}(t,x)\xi_i\xi_j \quad \text{ and } \quad
		b(t,x,\xi)  = i\sum_{j=1}^{n} b_{j}(t,x)\xi_j + b_{n+1}(t,x).
		\]
	\end{linenomath*}
	Here, the matrix $(a_{i,j}(t,x))$ is real symmetric for all $(t,x)\in (0,T] \times \R^n$, $a_{i,j} \in L^1((0,T];C^\infty(\R^n)) \cap C^1((0,T];C^\infty(\R^n))$ and $b_j \in L^1((0,T];C^\infty(\R^n))$. We have the following assumptions on $a(t,x,\xi)$, $b(t,x,\xi)$ and $b_j(t,x), j=0,n+1$ :
	\begin{linenomath*}
		\begin{equation}
			\label{conds}
			\begin{rcases}
				\begin{aligned}
					a(t,x,\xi) &\geq C_0 \o^2 \japxik^2, \quad C_0>0, \\
					\vert \partial_\xi^\alpha \partial_x^\beta a(t,x,\xi) \vert &\leq C^{|\alpha|+|\beta|} \alpha! (\beta!)^{\sigma} \frac{1}{t^{p}} \o^2 \P^{-\vert \beta \vert}\japxik^{2-\vert \alpha \vert},\\
					\vert \partial_\xi^\alpha \partial_x^\beta \partial_t a(t,x,\xi) \vert & \leq C^{|\alpha|+|\beta|} \alpha! (\beta!)^{\sigma} \frac{1}{t^{q}} \o^2 \P^{-\vert \beta \vert} \japxik^{2-\vert \alpha \vert}, \qquad \\
					\vert \partial_\xi^\alpha \partial_x^\beta b(t,x,\xi) \vert &\leq C^{|\alpha|+|\beta|} \alpha! (\beta!)^{\sigma}  \frac{1}{t^{r}}  \o \P^{-\vert \beta \vert}\japxik^{1-\vert \alpha \vert},\\
					|\partial_x^\beta b_j(t,x)| &\leq C^{|\beta|} (\beta!)^{\sigma} \frac{1}{t^r} \P^{-|\beta|}, \quad j=0,n+1,
				\end{aligned}
			\end{rcases}
		\end{equation}
	\end{linenomath*}		
	$q \in \left(1,\frac{3}{2}\right), p \leq q-1, p \in \big[0,\frac{1}{2} \big), r \in \big[0,1\big), 3 \leq \sigma < (q-p)/(q-1)$ and $ (t,x,\xi) \in [0,T] \times \R^n\times \R^n$. Note that $C>0$ is a generic constant.

	We now state the main result of this paper.
	Let $e=(1,1).$
	
	\begin{thm}\label{result1}
		Consider the strictly hyperbolic Cauchy problem  (\ref{eq1}) satisfying the conditions in (\ref{conds}). Let the initial data $f_j$ belong to $H^{s+(2-j)e, \Lambda_1,\sigma}_{\Phi,k}$ and the right hand side $f \in C([0,T];H^{s,\Lambda_2,\sigma}_{\Phi,k}),$ $\Lambda_j>0,j=1,2$.
		Then, there exist a continuous function $\Lambda(t)$ and positive constants $ \Lambda_0$ and $\delta^*,$ such that there is a unique solution
		\begin{linenomath*}
			\[
			u \in C\left([0,T];H^{s+e,\Lambda(t),\sigma}_{\Phi,k}\right)\bigcap C^{1}\left([0,T];H^{s,\Lambda(t),\sigma}_{\Phi,k}\right),
			\]
		\end{linenomath*}
		for $\Lambda(t) < \Lambda^*= \min\{ \Lambda_0,  \Lambda_1,  \Lambda_2 \}.$ More specifically, the solution satisfies an a priori estimate		
		\begin{linenomath*}
			\begin{equation}
				\begin{aligned}
					\label{est2}
					\sum_{j=0}^{1} \Vert \partial_t^ju(t,\cdot) &\Vert_{\Phi,k;s+(1-j)e, \Lambda(t),\sigma} \\
					&\leq C \Bigg(\sum_{j=1}^{2} \Vert f_j\Vert_{\Phi,k;s+(2-j)e,\Lambda(0),\sigma} 
					+ \int_{0}^{t}\Vert f(\tau,\cdot)\Vert_{\Phi,k;s,\Lambda(\tau),\sigma}\;d\tau\Bigg)
				\end{aligned}
			\end{equation}
		\end{linenomath*}
		where $0 \leq t \leq T \leq (\delta^*\Lambda^*/\lambda)^{1/\delta^*}, \; C=C_s>0$ and $\Lambda(t)= \frac{\lambda}{\delta^*} \left( T^{\delta^*} - t^{\delta^*}\right)$ for a sufficiently large $\lambda$.				
	\end{thm}
	
	\begin{rmk}
		Observe that we have $3 \leq \sigma < (q-p)/(q-1)$ where as in \cite[Theorem 2]{CSK}, it is $1 \leq \sigma < (q-p)/(q-1)$. The increase in the lower bound for $\sigma$ is due to the application of sharp G\r{a}rding inequality in our context  that dictates $\sigma \geq 3$. This is discussed in Section \ref{energy}. Due to this increment in $\sigma$, we have $ q \in  \Big(1, \frac{3}{2} \Big).$
	\end{rmk}
	
	\section{Parameter Dependent Global Symbol Classes}	\label{Symbol classes}
	We now define certain parameter dependent global symbols that are associated with the study of the Cauchy problem (\ref{eq1}). Let $m=(m_1,m_2)\in \mathbb{R}^2$. Consider the metrics $g_{\Phi,k}$, $\tilde g_{\Phi,k}^{(1)}$ and $\tilde g_{\Phi,k}^{(2)}$ as in (\ref{m1}) (\ref{me2}) and (\ref{me3}). These metrics can be related to the metrics of the form (\ref{m2}) where $(\rho,\tilde\rho)$ are $((1,0),(1,0)),$ $((1-p/(q-p),p/(q-p)),(1-p/(q-p),p/(q-p)))$ and $((1-1/\sigma, 1/\sigma),(1-1/\sigma, 1/\sigma)).$
	
	\begin{defn}
		$G^{m_1,m_2}(\omega,g_{\Phi,k}^{\rho,\tilde\rho})$ is the space of all functions $a \in C^\infty(\mathbb{R}^{2n})$ satisfying 
		\begin{linenomath*}
			\begin{equation}
				\label{sym1}
				   |\partial_\xi^\alpha  D_x^\beta a(x,\xi)| \leq C_{\alpha\beta} \japxik^{m_1-\rho_1|\alpha|+\rho_2|\beta|} \o^{m_2} \P^{-\tilde\rho_1|\beta|+\tilde\rho_2|\alpha|}
			\end{equation}	
		\end{linenomath*}		
	\end{defn}
	Since $g_{\Phi,k} \leq \tilde g_{\Phi,k}^{(1)} \leq \tilde g_{\Phi,k}^{(2)},$ we have
	\[
		G^{m_1,m_2}(\omega,g_{\Phi,k}) \subset G^{m_1,m_2}(\omega, \tilde g_{\Phi,k}^{(1)}) \subset G^{m_1,m_2}(\omega,\tilde g_{\Phi,k}^{(2)} ).
	\]
	Let $\mu \geq 1$ and $\nu \geq 1$.
		\begin{defn}
		$AG_{\mu,\nu}^{m_1,m_2}(\omega,g_{\Phi,k}^{\rho,\tilde\rho})$ is the space of all functions $a \in C^\infty(\mathbb{R}^{2n})$ satisfying (\ref{sym1}) with $C_{\alpha\beta} = C^{|\alpha| + |\beta|} (\alpha!)^{\mu} (\beta!)^{\nu}$ for some $C>0.$ 
	\end{defn}
	
	\begin{defn}
		$AG_{\sigma}^{m_1,m_2}(\omega,g_{\Phi,k}^{\rho,\tilde\rho})$ is the space of all functions $a \in C^\infty(\mathbb{R}^{2n})$ satisfying (\ref{sym1}) when $\h \leq C_1|\alpha|^{-\sigma}$ with $C_{\alpha\beta} = C_2^{|\alpha| + |\beta|} (\alpha!) (\beta!)^{\sigma}$ for some positive constants $C_1,C_2>0$
	\end{defn}
	
	We denote the set of operators with symbols in $G^{m_1,m_2}(\omega,g_{\Phi,k}^{\rho,\tilde\rho})$ and $AG_{\sigma}^{m_1,m_2}(\omega,g_{\Phi,k}^{\rho,\tilde\rho})$ by $OPG^{m_1,m_2}(\omega,g_{\Phi,k}^{\rho,\tilde\rho})$ and $OPAG_{\sigma}^{m_1,m_2}(\omega,g_{\Phi,k}^{\rho,\tilde\rho}),$ respectively. As far as the calculi of these pseudodifferential operators are concerned we refer to \cite[Appendix II \& III]{RahulNUK2}, \cite[Section 6.3]{nicRodi} and \cite[Appendix]{AscaCappi2}.
	
	In our analysis, we require the following conjugation result.
	\begin{prop}\label{conju}
		Let $e^{\Lambda(t) \Theta(x,D_x)} $ be as in (\ref{iopdo}) and $a(x,\xi) \in AG_{\sigma,\sigma}^{m_1,m_2}(\omega,\tilde g_{\Phi,k}^{(1)})$. Then, there exists $\Lambda_{0}>0$ such that for $\Lambda(t)>0$ with $\Lambda(t)<\Lambda_{0}$,
		$$
			e^{\Lambda(t) \Theta(x,D_x)} a(x, D_x) e^{-\Lambda(t) \Theta(x,D_x)}
			=a(x, D_x) + \sum_{j=1}^{3} r_{\Lambda}^{(j)}\left(t, x, D_{x}\right)
		$$
		where the symbols of $r_{\Lambda}^{(j)}\left(t, x, D_{x}\right)$ for $j=1,2,3$ are  in $C([0, T]; AG_{\sigma,\sigma}^{-\infty, m_{2}-\gamma} (\omega,\tilde g_{\Phi,k}^{(2)}) ),$ $ C([0, T]; AG_{\sigma,\sigma}^{ m_{1}-\gamma,-\infty} (\omega,\tilde g_{\Phi,k}^{(2)}) ),$ and $ C([0, T]; AG_{\sigma,\sigma}^{ -\infty,-\infty} (\omega,\tilde g_{\Phi,k}^{(2)}) ),$ respectively.
	\end{prop}
	\begin{proof}
		Noting the fact that $\tilde g_{\Phi,k}^{(1)} \leq \tilde g_{\Phi,k}^{(2)}$, the proof follows in similar lines to \cite[Theorem 4.0.1]{Rahul_NUK2}. 
	\end{proof}

	Observe that the derivatives of $\sqrt{a(t,x,\xi)},$ characteristic roots of operator $P$ in (\ref{eq1}) show stronger singular behaviour compared to $a(t,x,\xi)$ due to the singularity. Thus, to handle the singular behavior of the characteristics, we have the following symbol classes.
	\begin{defn}
		$AG_{\sigma}^{m_1,m_2}\{l;\delta_1\}^{(1)}_{N}(\omega,g_{\Phi,k}^{\rho,\tilde\rho})$ for $l \in \R$ and $\delta_1\in[0,1)$is the space of all functions $a \in C^1((0,T];G^{m_1,m_2}(\omega,g_{\Phi,k}^{\rho,\tilde\rho}))$ satisfying 
		\begin{linenomath*}
			\begin{equation}\label{sym2}
				\vert \partial_\xi^\alpha D_x^\beta a(t,x,\xi) \vert  \leq C^{|\alpha|+ |\beta|} \alpha! (\beta!)^{\sigma} \japxik^{m_1-\rho_1|\alpha|+\rho_2|\beta|} \o^{m_2} \P^{-\tilde\rho_1|\beta|+\tilde\rho_2|\alpha|} \left(\frac{1}{t}\right) ^{\delta_1l} ,
			\end{equation}
		\end{linenomath*}
		for all $(t,x,\xi)\in \pd $ and for some $C>0 $ where $\alpha, \beta \in \mathbb{N}^n_0.$  
	\end{defn}	
	
	\begin{defn}
		$AG_{\sigma}^{m_1,m_2}\{l_1,l_2,l_3;\delta_1,\delta_2\}^{(2)}_{N}(\omega,g_{\Phi,k}^{\rho,\tilde\rho})$ for $l_1.l_3\in \R, l_2 \in \{0,1\}$ and $\delta_1\in[0,1),\delta_2 \in (1,3/2)$ is the space of all functions $a \in C^1((0,T]; G^{m_1,m_2}(\omega,g_{\Phi,k}^{\rho,\tilde\rho}))$ satisfying 
		\begin{linenomath*}
			\begin{equation}\label{sym3}
				\begin{aligned}
				 \small	\vert \partial_\xi^\alpha D_x^\beta a(t,x,\xi) \vert 
				&\leq C^{|\alpha|+ |\beta|} \alpha! (\beta!)^{\sigma}  \japxik^{m_1-\rho_1|\alpha|+\rho_2|\beta|} \o^{m_2}\\ & \qquad \P^{-\tilde\rho_1|\beta|+\tilde\rho_2|\alpha|} \bigg(\frac{1}{t}\bigg)^{ \delta_1(l_1+l_2(|\alpha|+|\beta|)) + \delta_2l_3} 
				\end{aligned}
			\end{equation}
		\end{linenomath*}
		for all $(t,x,\xi)\in \hyp$ and for some $C_{\alpha \beta}>0$ where $\alpha,\beta \in \mathbb{N}^n_0.$ 
	\end{defn}	

	Given a $t$-dependent global symbol $a(t,x,\xi)$, we can associate a pseudodifferential operator $Op(a)=a(t,x,D_x)$ to $a(t,x,\xi)$ by the following oscillatory integral
	\begin{linenomath*}
		\begin{align*}
			a(t,x,D_x)u(t,x)& =\iint\limits_{\mathbb{R}^{2n}}e^{i(x-y)\cdot\xi}a(t,x,\xi){u}(t,y)dy \textit{\dj}\xi \\
			& = (2\pi)^{-n}\int\limits_{\mathbb{R}^n}e^{ix\cdot\xi}a(t,x,\xi)\hat{u}(t,\xi) \textit{\dj}\xi,
		\end{align*}
	\end{linenomath*}
	where $\textit{\dj}\xi = (2 \pi)^{-n}d\xi$ and $\hat u$ is the Fourier transform of $u$ in the space variable.
	The calculus for the operators with symbols of form $a(t,x,\xi) = a_1(t,x,\xi) + a_2(t,x,\xi)$ such that
	\begin{linenomath*}
		\[
		\begin{aligned}
			a_1 &\in AG_{\sigma}^{\tilde m_1,\tilde m_2}\{\tilde l;\delta_1\}_{N_1}^{(1)} \wmt, \\
			a_2 &\in AG_{\sigma}^{m_1,m_2}\{l_1,l_2,l_3;\delta_1,\delta_2\}_{N_2}^{(1)}\wmt,
		\end{aligned}
		\]
	\end{linenomath*}
	for $N_1 \geq N_2,$ can be readily built by following the
	similar standard arguments given in \cite[Appendix]{RahulNUK3}, \cite[Appendix II \& III]{RahulNUK2} and \cite[Appendix]{AscaCappi2}.
	
	\section{Proof of Theorem \ref{result1}} \label{Proof1}
	In this section, we give a proof of the main result, Theorem \ref{result1}. There are three key steps in the proof. 
	First, we factorize the operator $P(t,x,\partial_t,D_x)$. To this end, we begin with modifying the coefficients of the principal part by performing an excision so that the resulting coefficients are regular at $t=0$. Second, we reduce the original Cauchy problem to a Cauchy problem for a first order system (with respect to $\partial_t$). Lastly, using sharp G\r{a}rding’s inequality we arrive at the $L^2$ well-posedness of a related auxiliary Cauchy problem, which gives well-posedness of the original problem in the  weighted Sobolev spaces $H^{s,\varepsilon,\sigma}_{\Phi,k}$.
	
	\subsection{Factorization}\label{factr}
	
	Consider the operator $a(t,x,D_x)$ defined in (\ref{eq1}). We modify its symbol $a(t,x,\xi)$ in $Z_{int}(2)$, by defining
	\begin{linenomath*}
		\begin{equation}\label{exci}
			\tilde{a}(t,x,\xi)=\varphi(t\P \japxik)\o^{2} \japxik^{2} + (1-\varphi(t\Phi(x) \japxik))a(t,x,\xi)
		\end{equation}
	\end{linenomath*} 
	for $
	\varphi \in C^\infty(\mathbb{R}) \text{ , }0\leq \varphi \leq 1 \text{ , } \varphi=1 \text{ in }[0,1] \text{ , }\varphi=0 \text{ in }[2,+\infty).$
	Note that $(a-\tilde{a}) \in AG_{\sigma}^{2,2}\{1;p\}_{int,2}(\omega,g_{\Phi,k})$ and $(a-\tilde{a}) \sim 0$ in $Z_{ext}(2).$ This
	implies that $t^{p}(a-\tilde{a})$ for $ t \in [0,T]$ is a bounded and continuous family in $AG_{\sigma}^{2,2}(\omega,g_{\Phi,k})$. 
	Observe that $a-\tilde{a}$ is $L^1$ integrable in $t$, i.e.,
	\begin{equation}\label{diff}
		\begin{aligned}
			\int^{T}_{0}\vert (a-\tilde{a})(t,x,\xi)\vert dt 
			&\leq \kappa_{0}'  \o^{2} \japxik^{2} \int_{0}^{(2/\Phi(x) 	\japxik)^{1/(q-p)}}  \frac{1}{t^p}  dt	\\
			&\leq (\Phi(x) \japxik)^{2-(1-p)/(q-p)}.
		\end{aligned}
	\end{equation}
	as $\o \lesssim \P.$
	
	Let $\tau(t,x,\xi)= \sqrt{\tilde{a}(t,x,\xi)}$. Denote the indicator functions for the regions $Z_{int}(N_1)$ and $Z_{ext}(N_2)$ by $\chi_{1}(N_1)$ and  $\chi_{2}(N_2)$, respectively. It is easy to note that 
	\begin{enumerate}[label=\roman*)]
		\item $\tau(t,x,\xi)$ is $G_\omega$-elliptic symbol of order $(1,1)$ i.e. there is $C>0$ such that for all $(t,x,\xi)\in [0,T] \times \mathbb{R}^n \times \mathbb{R}^n$ we have
		\begin{linenomath*}
			\[
			\vert\tau(t,x,\xi)\vert \geq C  \o \japxik.
			\]
		\end{linenomath*}
		\item $\tau \in AG_{\sigma}^{1,1}\{0;0\}_{1}^{(1)}(\omega,g_{\Phi,k}) + AG_{\sigma}^{1,1}\{1/2,1,0;p,0\}_{1}^{(2)}(\omega,g_{\Phi,k}) $. More precisely, for $|\alpha| + |\beta| >0,$
		\begin{equation}\label{root1}
			\begin{rcases}
				\begin{aligned}
					|\tau(t,x,\xi)|  & \leq C_{0} \o \japxik \left(\chi_{1}(1) + \chi_{2}(1) t^{-p/2}\right),\\
					{|\partial_\xi^\alpha D_x^\beta \tau(t,x,\xi)|} & \leq C_{\alpha \beta}  \o \P^{-|\beta|} \japxik^{1-|\alpha|}  \left(  \chi_{1}(1) +  \chi_{2}(1)t^{-p(|\alpha|+|\beta|)} \right).
				\end{aligned}
			\end{rcases}
		\end{equation}
		\item $\partial_t \tau$ is such that for $|\alpha|+|\beta|>0$ we have{\small
			\begin{linenomath*}
				\begin{equation*}
					\begin{aligned}
						{|\partial_t\tau(t,x,\xi)|}  & \sim 0 \text{ in } Z_{int}(1),\\
						{|\partial_t\tau(t,x,\xi)|}  & \leq C_{0} \; \o \japxik  \left( \chi_{1}(2) \o \japxik t^{-p}+ \chi_{2}(1) t^{-q} \right),\\
						{|\partial_\xi^\alpha D_x^\beta \partial_t \tau(t,x,\xi)|} & \leq C_{\alpha \beta}  \o \P^{-|\beta|} \japxik^{1-|\alpha|}  \left( \chi_{1}(2)\o \japxik + \chi_{2}(1) t^{-q} \right)  t^{-p(|\alpha|+|\beta|)}.
					\end{aligned}
				\end{equation*}
		\end{linenomath*}}
		By the definition of the time splitting point and the subdivision of the phase space, we see that
		\begin{linenomath*}
			\begin{equation}
				\begin{rcases}
					\begin{aligned}
						{|\partial_t\tau(t,x,\xi)|}  & \sim 0 \text{ in } Z_{int}(1),\\
						{|\partial_\xi^\alpha D_x^\beta \partial_t \tau(t,x,\xi)|} & \leq C_{\alpha \beta} \chi_{1}(2) \o \P^{-|\beta|} \japxik^{1-|\alpha|}  t^{-q}  t^{-p(|\alpha|+|\beta|)},
					\end{aligned}
				\end{rcases}
			\end{equation}
		\end{linenomath*}
		for $|\alpha| + |\beta| \geq 0.$ Hence, $\partial_t\tau \sim 0$ in $ Z_{int}(1)$ and $\partial_t\tau \in AG_{\sigma}^{1,1}\{0,1,1;p,q\}_{1}^{(2)}(\omega,g_{\Phi,k}).$
	\end{enumerate}  
	
	From the above properties of $\tau$ and by the definition of $\tilde a$ in (\ref{exci}), we have the following two lemmas.
	\begin{lem}\label{lemroot}
		Let $\tilde g_{\Phi,k}^{(1)})$ be as in (\ref{me2}) and $\delta$ as in (\ref{gamma}). Then,
		\begin{enumerate}[label=\roman*)]
			\item $\tau \in L^\infty \left( [0,T]; AG_{\sigma}^{1+\delta'/2,1}(\omega\Phi^{\delta'/2}, \tilde g_{\Phi,k}^{(1)}) \right)$,  
			\item $t^{1-\frac{p}{2}}\tau \in C \left( [0,T]; AG_{\sigma}^{1,1}(\omega, \tilde g_{\Phi,k}^{(1)}) \right)$, 
			\item $\tau^{-1} \in C\big([0,T];  AG_{\sigma}^{-1,-1}(\omega,\tilde g_{\Phi,k}^{(1)}) \big)$,
			\item $t^{1-\delta}\partial_t \tau \in C\big([0,T];  AG_{\sigma}^{1+1/\sigma,1}(\omega\Phi^{1/\sigma},\tilde g_{\Phi,k}^{(1)})\big)$.
		\end{enumerate}	
	\end{lem}
	\begin{proof}
		The first claim follows from (\ref{root1}) and the observation that in $Z_{ext}(1)$
		\begin{equation}\label{est1}
			\left( \frac{1}{t} \right)^{p} \leq \left(  \frac{\P\japxik}{N} \right)^{p/(q-p)},
		\end{equation}
		while the second and third claims are straight forward consequences of (\ref{root1}) and (\ref{est1}).  The fourth claim follows from (\ref{est1}) and the following estimate in $Z_{ext}(1)$
		\[
		\begin{aligned}
			\frac{1}{t^q} &= \frac{1}{t^{1-\delta}} \; \frac{1}{t^{(q-p)/\sigma}} \leq \frac{1}{t^{1-\delta}} \; \left( \frac{\P\japxik}{N} \right)^{1/\sigma}.
		\end{aligned}
		\]
		
	\end{proof}
	
	\begin{lem}\label{lem2}
		Let $\delta$ as in (\ref{gamma}). Then,
		\begin{enumerate}[label=\roman*)]
			\item $t^{1-\delta} (a(t,x,D_x)-\tilde a(t,x,D_x)) \in C\Big([0,T];OPAG_{\sigma}^{1+1/\sigma,1+1/\sigma}(\omega,g_{\phi,k})\Big),$
			
			\item $\tilde{a}(t,x,D_x) -\tau(t,x,D_x) ^2 \in L^\infty\Big([0,T];OPAG_{\sigma}^{1,1}(\omega,\tilde g_{\Phi,k}^{(1)})\Big) ,$
			
			\item $t^{r} b(t,x,D_x) \in C\Big([0,T];OPAG_{\sigma}^{1,1}(\omega,g_{\phi,k})\Big)$
		\end{enumerate}
	\end{lem}
	\begin{proof}
		The proof is a consequence of the fact that in $Z_{int}(2)$
		\[
		\begin{aligned}
			|\partial_x^\beta \partial_\xi^\alpha (a-\tilde a)(t,x,\xi) | &\leq C_{\alpha\beta} \chi_1(2) \o^2 \P^{-|\beta|} \japxik^{2-|\alpha|} \frac{1}{t^p}\\
			& \leq C_{\alpha\beta} \chi_1(2) \o^{1+1/\sigma} \P^{-|\beta|} \japxik^{{1+1/\sigma}-|\alpha|}  (\P\japxik)^{^{1-1/\sigma}} \frac{1}{t^p}\\
			& \leq C_{\alpha\beta} \chi_1(2) \o^{1+1/\sigma} \P^{-|\beta|} \japxik^{{1+1/\sigma}-|\alpha|}  \frac{1}{t^{1-\delta-p}} \; \frac{1}{t^p}\\
			& \leq C_{\alpha\beta} \chi_1(2) \o^{1+1/\sigma} \P^{-|\beta|} \japxik^{{1+1/\sigma}-|\alpha|}  \frac{1}{t^{1-\delta}}.
		\end{aligned}
		\]
		The second and third claims follow directly from the definitions of $\tilde a(t,x,D_x)$ and $b(t,x,D_x).$
	\end{proof}
	
	Let us define $\delta^* >0$ as
	\begin{equation}\label{vare}
		\delta^* = \min\{ \delta, 1-r,1-p\}.
	\end{equation}
	We are interested in the factorization of the operator $P(t,x,\partial_t,D_x)$. This leads to
	\begin{linenomath*}
		\begin{equation*}
			P = (\partial_t-i\tau(t,x,D_x)) 	(\partial_t+i\tau(t,x,D_x))+ b_0(t,x)\partial_t + (a-\tilde{a}+ a_1)(t,x,D_x)
		\end{equation*}
	\end{linenomath*}
	where the operator $a_1(t,x,D_x)$ is such that, for $t \in [0,T]$, 
	\begin{linenomath*}
		\begin{equation*}
			a_1 = -i[\partial_t,\tau] + \tilde{a} -\tau^2 + b \; \text{ and } \; t^{1-\delta^*}a_1(t,x,D_x) \in OPAG_{\sigma}^{1+1/\sigma,1}(\omega\Phi^{1/\sigma},\tilde g_{\Phi,k}^{(1)}).
		\end{equation*}
	\end{linenomath*}
	
	\subsection{First Order Pseudodifferential System}
	We will now reduce the operator $P$ to an equivalent first order $2\times2$ pseudodifferential system. The procedure is similar to the one used in \cite{RahulNUK3,RahulNUK5,Cico1}. To achieve this, we introduce the change of variables $U=U(t,x)=(u_1(t,x),u_{2}(t,x))^T$, where
	\begin{equation}
		\label{COV}
		\begin{cases}
			u_1(t,x)= (\partial_t+i\tau(t,x,D_x))u(t,x), \\ 
			u_2(t,x)=  \o \la D_x \rak u(t,x) - H(t,x,D_x)u_1, \\  
		\end{cases}
	\end{equation}
	and the operator $H$ with the symbol $\sigma(H)(t,x,\xi)$ is such that
	\begin{linenomath*}
		\[
		\sigma(H)(t,x,\xi) = -\frac{i}{2}\o \japxik  \frac{\Big(1-\varphi\Big(t\P \japxik/3\Big)\Big)}{\tau(t,x,\xi)}.
		\]
	\end{linenomath*}
	Note that by the definition of $H$,  $\text{supp } \sigma(H) \cap \text{supp } \sigma(a-\tilde{a}) = \emptyset$ and we have 
	\begin{linenomath*}
		\begin{align*}
			\sigma(2iH(t,x,D_x) \circ \tau(t,x,D_x) )&\sim 0,  \quad \text{ in } Z_{int}(3),\\
			\sigma(2iH(t,x,D_x) \circ \tau(t,x,D_x) )&= \o \la \xi \ra_k (1+\sigma(K_1)), \quad \text{ in } Z_{ext}(3),
		\end{align*}
	\end{linenomath*}
	where $\sigma(K_1) \in AG_{\sigma}^{-1,-1}\{0;p\}_{6}^{(1)}(\omega,g_{\Phi,k}) + AG_{\sigma}^{-1,-1}\{2,1,0;p,q\}_{3}^{(2)}(\omega,g_{\Phi,k})$. Then, the equation $Pu=f$ is equivalent to the first order $2\times2$ system  :
	\begin{equation}
		\label{FOS1}
		\begin{aligned}
			LU &= (\partial_t - \mathcal{D}+A_0+A_1)U=F,\\
			U(0,x)&=(f_2+i\tau(0,x,D_x)f_1,\P\la D_x\ra f_1)^T ,
		\end{aligned}
	\end{equation}
	where
	\begin{linenomath*}
		\begin{align*}
			F&=(f(t,x) ,-H(t,x,D_x)f(t,x) )^T,\\
			\mathcal{D} &= \text{diag}(i\tau(t,x,D_x),-i\tau(t,x,D_x)),\\
			A_0 &= \begin{pmatrix}
				B_0H & B_0 \\
				-HB_0H & \quad HB_0
			\end{pmatrix}
			= \begin{pmatrix}
				\mathcal{R}_1 & B_0 \\
				-\mathcal{R}_3 & \mathcal{R}_2
			\end{pmatrix},\\
			A_1 &= \begin{pmatrix}
				B_1H + B_3& B_1 +B_4\\
				B_2 -HB_3& \qquad i[M,\tau]M^{-1}-H(B_1 + B_4)
			\end{pmatrix}.
		\end{align*}
	\end{linenomath*}
	The operators $M, M^{-1},B_0,B_1$ and $B_2$ are as follows
	\begin{linenomath*}
		\begin{align*}
			M &= \o\la D_x \rak, \quad
			M^{-1} = \la D_x \rak^{-1}  \o^{-1},  \\
			B_0 &= (a(t,x,D_x) -\tilde a (t,x,D_x))\la D_x \rak^{-1}  \o^{-1},  \\
			B_1 &= (-i\partial_t\tau(t,x,D_x) + \tilde{a}(t,x,D_x) -\tau(t,x,D_x)^2 + b(t,x,D_x)) \la D_x \rak^{-1}  \o^{-1},\\
			B_2 &= 2iH\tau-M+i[M ,\tau]M^{-1}H + i[\tau, H]-HB_1H+\partial_tH\\
			B_3 &= b_0(1-i\lambda M^{-1} H), \qquad B_4 = ib_0\lambda M^{-1}.
		\end{align*}
	\end{linenomath*}
	Here $b_0=b_0(t,x)$ is as in (\ref{eq1}).
	By the definition of operator $H$, we have $B_0H = \mathcal{R}_1, HB_0 =\mathcal{R}_2$, $HB_0H=\mathcal{R}_3$ for $\mathcal{R}_j\in G^{-\infty,-\infty}(\omega, g_{\Phi,k}),j=1,2,3,$ and the operator $2iH\tau-M$ is such that
	\begin{linenomath*}
		\[
		\sigma(2iH\tau-M) = \begin{cases}
			-\o\la \xi \rak, & \text{ in } Z_{int}(3),\\
			\o\la \xi \rak \sigma(K_1), & \text{ in } Z_{ext}(3).
		\end{cases}
		\]
	\end{linenomath*}
	Since $2p \leq q,$ we have
	\[
	AG_{\sigma}^{0,0} \{2,1,0;p,q\}_N^{(2)} (\omega,{g}_{\Phi,k}) \subset AG_{\sigma}^{0,0} \{0,1,1;p,q\}_N^{(2)} (\omega,{g}_{\Phi,k}).
	\]
	The symbols of operators $\mathcal{D}, A_0$ and $A_1$ are in the following symbol classes
	\begin{linenomath*}
		\begin{equation}\label{As2}
			\begin{rcases}
				\begin{aligned}
					\sigma(\mathcal{D}) &\in AG_{\sigma}^{1,1}\{0;0\}_{2}^{(1)}(\omega,g_{\Phi,k}) + AG_{\sigma}^{1,1}\{1,1,0;p,0\}_{1}^{(2)}(\omega,g_{\Phi,k}) \\
					\sigma(A_0) &\in AG_{\sigma}^{1,1}\{1;p\}_{2}^{(1)}(\omega,{g}_{\Phi,k}) + AG_{\sigma}^{-\infty,-\infty}\{0,0,0;0,0\}_{3}^{(2)}(\omega,{g}_{\Phi,k}), \\
					\sigma(A_1) & \in AG_{\sigma}^{1,1}\{0;0\}_{6}^{(1)}(\omega,{g}_{\Phi,k}) + AG_{\sigma}^{0,0}\{1;r\}_{1}^{(1)}(\omega,{g}_{\Phi,k})\\ & \qquad+AG_{\sigma}^{0,0}\{0,1,1;p,q\}_{1}^{(2)}(\omega,{g}_{\Phi,k})
				\end{aligned}
			\end{rcases}
		\end{equation}
	\end{linenomath*}
	and thus, by Lemmas \ref{lemroot} - \ref{lem2} and the choice of $\delta^*$ as in (\ref{vare}), 
	\begin{linenomath*}
		\begin{equation}\label{As}
			\begin{rcases}
				\begin{aligned}
					t^{1-\delta^*}\sigma(A_0(t)) &\in C\left([0,T]; AG_{\sigma}^{1/\sigma,1/\sigma}(\omega,{g}_{\Phi,k}) \right),\\
					t^{1-\delta^*}\sigma(A_1(t)) &\in 	C\left([0,T];AG_{\sigma}^{1/\sigma,1/\sigma}(\omega,\tilde{g}_{\Phi,k}^{(1)} \right).
				\end{aligned}
			\end{rcases}
		\end{equation}
	\end{linenomath*}
	As ${g}_{\Phi,k}) \leq \tilde{g}^{(1)}_{\Phi,k}) \leq \tilde{g}^{(2)}_{\Phi,k}),$ from (\ref{As}) we have
	\begin{equation}\label{As2}
		t^{1-\delta^*}\sigma(A_0(t)),\; t^{1-\delta^*}\sigma(A_1(t)) \in  C\left([0,T];AG_{\sigma}^{1/\sigma,1/\sigma}(\omega,\tilde{g}_{\Phi,k}^{(2)} \right).
	\end{equation}
	Let us choose $\lambda>0$ as large as possible so that
	\begin{equation}\label{lam}
		|\sigma(A_0(t))| +| \sigma(A_1(t))| \leq \frac{\lambda}{t^{1-\delta^*}} (\P\japxik)^{1/\sigma}.
	\end{equation}
	
	\subsection{Energy Estimate} \label{energy}
	In this section, we prove the estimate (\ref{est2}). Note that it is sufficient to consider the case $s=(0,0)$ as the operator $\P^{s_2} \la D\ra^{s_1}L\la D\ra^{-s_1}\P^{-s_2}$, where $s=(s_1,s_2)$ is the index of the weighted Sobolev space, satisfies the same hypotheses as $L$.
	
	In the following, we establish some lower bounds for the operator $\mathcal{D} - A_0 - A_1$. The symbol $d(t,x,\xi)$ of the operator $\mathcal{D}(t)+\mathcal{D}^*(t)$ is such that
	\begin{linenomath*}
		\[
		d \in AG_{\sigma}^{0,0}\{0;0\}_{2}^{(1)}(\omega,g_{\Phi,k}) + AG_{\sigma}^{0,0}\{1/2,1,0;p,q\}_{1}^{(2)}(\omega,g_{\Phi,k}).
		\]
	\end{linenomath*}
	It follows from the definition of $\delta^*$ and Lemma \ref{lemroot} that 
	\begin{linenomath*}
		\[
		t^{1-\delta^*}d \in C([0,T];AG_{\sigma}^{0,0}(\omega,\tilde g_{\Phi,k}^{(1)}).
		\]
	\end{linenomath*}
	Thus
	\begin{equation}\label{lb1}
		2\Re \la \mathcal{D}{U},{U} \ra_{L^2} \geq -\frac{C_1}{t^{1-\delta^*}} \la {U},{U} \ra_{L^2}, \quad C_1>0.
	\end{equation}
	
	To control lower order terms, we make the following change of variable 
	\begin{equation}\label{c1}
		V(t,x)= e^{\Lambda(t)\Theta(x,D_x)}{U}(t,x),
	\end{equation}
	where $ \Lambda(t) = \frac{\lambda}{\delta^*}(T^{\delta^*}-t^{\delta^*})$ with $\lambda$ as in (\ref{lam}), and the operator $\Theta(x,D_x)$ is as in (\ref{iopdo}). From \cite[Corollary 4.0.4]{Rahul_NUK2}, 
	\[
		  e^{\pm\Lambda(t)\Theta(x,D_x)} e^{\mp\Lambda(t)\Theta(x,D_x)} = I+R^{(\pm)}(t,x,D_x),
	\]
	where for sufficiently large $k$ the operators $I+R^{(\pm)}(t,x,D_x)$ are invertible. Let us denote the operators $I+R^{(+)}(t,x,D_x),$ $I+R^{(-)}(t,x,D_x)$ and $e^{\pm\Lambda(t)\Theta(x,D_x)}$ by $\mathcal{R}(t,x,D_x)$, $\tilde{\mathcal{R}}(t,x,D_x)$ and $E^{(\pm)}(t,x,D_x),$ respectively. Then, from (\ref{c1}),
	\[
		U(t,x) = \mathcal{R}^{(-)}(t,x,D_x) e^{-\Lambda(t)\Theta(x,D_x)} V(t,x) \text{ and }\;  \|{U}(t)\|_{\Phi, k ; s, \Lambda(t), \sigma} = \|V(t)\|_{L^{2}}.
	\]
	Then $P u=f$ is equivalent to $L_{1} V=F_{1}$ where
	$$
	L_{1}=\partial_{t} - \mathcal{D}+\left(B+\frac{\lambda}{t^{1-\delta}} \Theta(x,D_x)\right)
	$$
	$F_{1}(t, x)=\mathcal{R}^{-1} E^{(+)} \tilde{\mathcal{R}} F(t, x)$ and the operator $B\left(t, x, D_{x}\right)$ is given by
	$$
	B=\mathcal{R}^{-1} E^{(+)}\left(\tilde{\mathcal{R}}\left(\partial_{t} \tilde{\mathcal{R}}^{-1}\right)+\tilde{\mathcal{R}} A \tilde{\mathcal{R}}^{-1}\right) E^{(-)}-\left(\mathcal{R}^{-1} E^{(+)} \tilde{\mathcal{R}} \mathcal{D} \tilde{\mathcal{R}}^{-1} E^{(-)}-\mathcal{D}\right)
	$$
	Observe that from Proposition \ref{conju} and from the Cauchy data given in conditions of Theorem \ref{result1}, we need $\Lambda(t)<\Lambda^{*}=\min \left\{\Lambda_{0}, \Lambda_{1}, \Lambda_{2}\right\}$. This implies $T<\left(\frac{\delta^*}{\lambda} \Lambda^{*}\right)^{1 / \delta^*}$. Then, we have $t^{1-\delta^*} B \in C\left([0, T] ; O P A G_{\sigma}^{\frac{1}{\sigma} e} \{\omega, \tilde g^{(2)}_{\Phi,k}\} \right) .$ 
	Choosing $\lambda$ sufficiently large, we obtain 
	\begin{equation}\label{est3}
		\Re\left\langle\left(\frac{\lambda}{t^{1-\delta^*}}\Theta(x,D_x)+B\right) V, V\right\rangle_{L^{2}} \geq-C_{2}\|V\|_{L^{2}}, \quad C_{2}>0.
	\end{equation}
	The above estimate is the result of application of sharp Gårding inequality, see \cite[Theorem 18.6.14]{Horm} for the metric $\tilde{g}_{\Phi, k}^{(2)}$ with the Planck function $(\Phi(x)\langle\xi\rangle_{k})^{\frac{1}{\sigma}-\gamma}$. It is important to note that the application of sharp G\r{a}rding inequality requires $\sigma \geq 3$. 
	
	The estimate (\ref{est2}) on the solution $u$ can be established by proving that the function $V(t, x)$ satisfies the a priori estimate
	\begin{equation}\label{est4}
		\|V(t)\|_{L^{2}}^{2} \leq C\left(\|V(0)\|_{L^{2}}^{2}+\int_{0}^{t}\left\|F_{1}(\tau, \cdot)\right\|_{L^{2}} d \tau\right), \quad t \in[0, T], C>0.
	\end{equation}
	The Cauchy problem for the operator $L_{1}$ is given by 
	\begin{equation}\label{FOS2}
		\begin{rcases}
		\begin{aligned}
		\partial_{t} V(t,x)&=\left( \mathcal{D} - \frac{\lambda}{t^{1-\delta}} \Theta(x,D_x)  -B 	\right)V(t,x)+F_{1}(t, x), \\
		V(0,x) & = e^{\Lambda(0)\Theta(x,D_x)}U(0,x)
		\end{aligned}
	\end{rcases}
	\end{equation}
	Observe that
	$$
	\begin{aligned}
		\partial_{t}\|V(t)\|_{L^{2}}^{2}=& 2 \Re\la \partial_{t} V, V\ra_{L^{2}} \\
		=& 2 \operatorname{Re}\la\mathcal{D} V, V \ra_{L^{2}}-2 \Re\left\langle\left(\frac{\lambda}{t^{1-\delta}}\left(\Phi(x)\left\langle D_{x}\right\rangle_{k}\right)^{1 / \sigma}+B\right) V, V\right\rangle_{L^{2}} \\
		&+2 \Re\la F_{1}, V \ra.
	\end{aligned}
	$$
	From (\ref{lb1}) and (\ref{est3}) we have
	$$
	\frac{d}{d t}\|V(t)\|_{L^{2}}^{2} \leq \frac{C}{t^{1-\delta^*}}\|V(t)\|_{L^{2}}+C\left\|F_{1}(t, \cdot)\right\|_{L^{2}}
	$$
	Considering the above inequality as a differential inequality, we apply Gronwall's lemma and obtain that
	$$
	\|V(t)\|_{L^{2}}^{2} \leq C' e^{\frac{T^{\delta^*}}{\delta^*}} \left( \|V(0)\|_{L^{2}}^{2} +  \int_{0}^{t}\left\|F_{1}(\tau, \cdot)\right\|_{L^{2}}^{2} d \tau \right)
	$$
	This proves the well-posedness of the Cauchy problem (\ref{FOS2}). Note that the solution ${U}$ to (\ref{FOS1}) belongs to $C\left([0, T] ; H_{\Phi, k}^{s, \Lambda(t), \sigma}\right)$. Returning to our original solution $u=u(t, x)$ we obtain the estimate (\ref{est2}) with

	\begin{linenomath*}
		\[
		u \in C([0,T];H^{s+e,\Lambda(t),\sigma}_\Phi) \cap C^{1}([0,T];H^{s, \Lambda(t) \sigma}_\Phi).
		\]
	\end{linenomath*}
	This concludes the proof.
	
	\section{Cone Condition}\label{cone}
	Existence and uniqueness follow from the a priori estimate established in the previous section. It now remains to prove the existence of cone of dependence. 
	
	We note here that the $L^1$ integrability of the characteristics plays a crucial role in arriving at the finite propagation speed. The implications of the discussion in \cite[Section 2.3 \& 2.5]{JR} to the global setting suggest that if the Cauchy data in (\ref{eq1}) is such that $f \equiv 0$ and $f_1,f_2$ are supported in the ball $\vert x \vert \leq R$, then the solution to Cauchy problem (\ref{eq1}) is supported in the ball $\vert x \vert \leq R+c^* \o t^{1-\frac{p}{2}}$. The constant $c^*$ is such that the quantity $c^* \o t^{-\frac{p}{2}}$ dominates the characteristic roots, i.e.,
	\begin{equation}
		\label{speed}
		c^*= \sup\Big\{\sqrt{a(t,x,\xi)}\o^{-1} t^{\frac{p}{2}}:(t,x,\xi) \in[0,T] \times \R^n_x \times \R^n_\xi,\:|\xi|=1\Big\}.
	\end{equation}
	Note that the support of the solution increases as $|x|$ increases since $\o$ is monotone increasing function of $|x|$. 
	
	In the following we prove the cone condition for the Cauchy problem $(\ref{eq1})$. Let $K(x^0,t^0)$ denote the cone with the vertex $(x^0,t^0)$:
	\begin{linenomath*}
		\[
		K(x^0,t^0)= \{(t,x) \in [0,T] \times \R^n : |x-x^0| \leq c^*\o (t^0-t)^{1-\frac{p}{2}}\}.
		\]
	\end{linenomath*}
	Observe that the slope of the cone is anisotropic, that is, it varies with both $x$ and $t$.
	
	\begin{prop}
		The Cauchy problem (\ref{eq1}) has a cone dependence, that is, if
		\begin{equation}\label{cone1}
			f\big|_{K(x^0,t^0)}=0, \quad f_i\big|_{K(x^0,t^0) \cap \{t=0\}}=0, \; i=1, 2,
		\end{equation}
		then
		\begin{equation}\label{cone2}
			u\big|_{K(x^0,t^0)}=0.
		\end{equation}
	\end{prop}
	\begin{proof}
		Consider $t^0>0$, $C^*>0$ and assume that  (\ref{cone1}) holds. We define a set of operators $P_\varepsilon(t,x,\partial_t,D_x), 0 \leq \varepsilon \leq \varepsilon_0$ by means of the operator $P(t,x,\partial_t,D_x)$ in (\ref{eq1}) as follows
		\begin{linenomath*}
			\[
			P_\varepsilon(t,x,\partial_t,D_x) = P(t+\varepsilon,x,\partial_t,D_x), \: t \in [0,T-\varepsilon_0], x \in \R^n,
			\]
		\end{linenomath*}
		and $\varepsilon_0 < T-t^0$, for a fixed and sufficiently small $\varepsilon_0$. For these operators we consider Cauchy problems
		\begin{linenomath*}
			\begin{alignat*}{2}
				P_\varepsilon v_\varepsilon & =f,  &&  t \in [0,T-\varepsilon_0], \; x \in \R^n,\\
				\partial_t^{k-1}v_\varepsilon(0,x)& =f_k(x),\qquad && k=1,2.
			\end{alignat*}
		\end{linenomath*}
		Note that $v_\varepsilon(t,x)=0$ in $K(x^0,t^0)$ and $v_\varepsilon$ satisfies an a priori estimate (\ref{est2}) for all $t \in[0,T-\varepsilon_0]$. Further, we have 
		\begin{linenomath*}
			\begin{alignat*}{2}
				P_{\varepsilon_1} (v_{\varepsilon_1}-v_{\varepsilon_2}) & = (P_{\varepsilon_2}-P_{\varepsilon_1})v_{\varepsilon_2},\qquad  &&  t \in [0,T-\varepsilon_0], \; x \in \R^n,\\
				\partial_t^{k-1}(v_{\varepsilon_1}-v_{\varepsilon_2})(0,x)& = 0,\qquad && k=1,2.
			\end{alignat*}
		\end{linenomath*}
		Since our operator is of second order, for the sake of simplicity we denote $b_{j}(t,x)$, the coefficients of lower order terms, as $a_{0,j}(t,x), 1 \leq j \leq n,$ while $b_0(t,x)$ and $b_{n+1}(t,x)$ are denoted as $a_{1,0}(t,x)$ and $a_{0,0}(t,x),$ respectively. Let $a_{i,0}(t,x) =0, \; 2 \leq i \leq n.$ Substituting $s-e$ for $s$ in the a priori estimate, we obtain
		\begin{equation}\label{cone3}
			\begin{aligned}
				&\sum_{j=0}^{1} \Vert  \partial_t^j(v_{\varepsilon_1}-v_{\varepsilon_2})(t,\cdot) \Vert_{\Phi,k;s-je,\Lambda(t),\sigma} \\
				&\leq C \int_{0}^{t}\Vert (P_{\varepsilon_2}-P_{\varepsilon_1})v_{\varepsilon_2}(\tau,\cdot)\Vert_{\Phi,k;s-e,\Lambda(\tau),\sigma}\;d\tau\\
				&\leq C \int_{0}^{t} \sum_{i,j=0 }^{n}\Vert (a_{i,j}(\tau+\varepsilon_1,x) - a_{i,j}(\tau+\varepsilon_2,x))D_{ij} v_{\varepsilon_2}(\tau,\cdot)\Vert_{\Phi,k;s-e,\Lambda(\tau),\sigma}\;d\tau,
			\end{aligned}
		\end{equation}
		where $D_{00}=I, D_{10}=\partial_t, D_{i0}=0,i \geq 2, D_{0j}=\partial_{x_j},j \neq 0$ and $D_{ij}=\partial_{x_i} \partial_{x_j}, i,j \neq 0$. Using the Taylor series approximation in $\tau$ variable, we have
		\begin{linenomath*}
			\begin{align*}
				|a_{i,j}(\tau+\varepsilon_1,x) - a_{i,j}(\tau+\varepsilon_2,x)| &= \Big|\int_{\tau+\varepsilon_2}^{\tau+\varepsilon_1} (\partial_ta_{i,j})(r,x)dr \Big|\\
				&\leq \o^{2} \Big|\int_{\tau+\varepsilon_2}^{\tau+\varepsilon_1}\frac{dr}{r^q}\Big|\\
				&\leq \o^{2}|E(\tau,\varepsilon_1,\varepsilon_2)|,
			\end{align*}
		\end{linenomath*}
		where
		\begin{linenomath*}
			\[
			E(\tau,\varepsilon_1,\varepsilon_2) = \frac{1}{q-1} \left(  (\tau+\varepsilon_1)^{-q+1} - (\tau+\varepsilon_2)^{-q+1} \right).
			\]
		\end{linenomath*}
		Note that $\o \lesssim \P$ and $E(\tau,\varepsilon,\varepsilon)=0$.
		Then right-hand side of the inequality in (\ref{cone3}) is dominated by
		\begin{linenomath*}
			\begin{equation*}\label{cone4}
				C \int_{0}^{t} |E(\tau,\varepsilon_1,\varepsilon_2)| \Vert  v_{\varepsilon_2}(\tau,\cdot)\Vert_{\Phi,k;s+e,\Lambda(\tau),\sigma}\;d\tau,
			\end{equation*}
		\end{linenomath*}
		where $C$ is independent of $\varepsilon$. By definition, $E$ is $L_1$-integrable in $\tau$.
		
		The sequence $v_{\varepsilon_k}$, $k=1,2,\dots$ corresponding to the sequence $\varepsilon_k \to 0$ is in the space
		\begin{linenomath*}
			\[
			C\Big([0,T^*];H^{s,\Lambda(t),\sigma}_{\Phi,k}\Big) \bigcap C^{1}\Big([0,T^*];H^{s-e,\Lambda(t),\sigma}_{\Phi,k}\Big), \quad T^*>0,
			\]
		\end{linenomath*}
		and $u=\lim\limits_{k\to\infty}v_{\varepsilon_k}$ in the above space and hence, in $\mathcal{D}'(K(x^0,t^0))$. In particular,
		\begin{linenomath*}
			\[
			\la u,\varphi\ra = \lim\limits_{k\to\infty} \la v_{\varepsilon_k}, \varphi \ra =0,\; \forall \varphi \in \mathcal{D}(K(x^0,t^0))
			\]
		\end{linenomath*}
		gives (\ref{cone2}) and completes the theorem. 
	\end{proof}

	\section{Counterexamples}\label{CE}
	
	In this section, we show by a set of counterexamples how the lower order terms of operator $P$ in (\ref{eq1}) influence the loss of regularity and well-posedness when $L^1$ integrability condition is violated. We cover various cases such as no loss, finite loss and nonuniqueness. All the examples in this section correspond to the case $\o=\P=1, \; x \in \R$ in (\ref{sing}).

	Following example shows that one can encounter finite loss when the coefficients of lower order terms are not $L^1$ integrable.
	\begin{exmp}
		\begin{equation}
			\begin{cases}
				\begin{aligned}
					&\left( \partial_t^2 - \partial_x^2 + \frac{1}{2t} \left( \partial_t -(4m+1)\partial_x \right) \right)u(t,x) =0, \\
					&\; u(0,x) = u_0(x), \;\; \partial_tu(0,x) = (4m+1)\partial_xu_0(x),
				\end{aligned}
			\end{cases}
		\end{equation}
	for some $m \in \mathbb{N}_0.$ Note that the above problem corresponds to the case $p=q=0$ and $r=1$ in (\ref{conds}). The solution to the above Cauchy problem is given by
	\[
		u(t,x) = \sum_{j=0}^{m} C_j^{(m)} t^j \partial_x^ju_0(x+t),
	\]
	for $C_j^{(m)}$ of the form
	\[
		C_0 =1, \quad C_j^{(m)} = \frac{(-2)^j}{j!} \; \frac{(m)_j}{\left(-\frac{1}{2}\right)_j}, \quad j \geq 1,
	\]
	where $(y)_j, y \in \R,$ is the $j^{th}$ falling factorial of $y$ \cite{ff} given by
	\[
		(y)_j = y(y-1) \cdots (y-j+1).
	\]
	For example, when $m=2,$ $u(t,x)$ is of the form
	\[
		u(t,x) = u_0(x+t) + 8t\partial_xu_0(x+t) + \frac{16}{3} t^2\partial_x^2u_0(x+t).
	\]
	We observe loss of derivatives: if $u_0 \in H^s,$ then $u(t,\cdot) \in H^{s-m}$ where $H^s$ is the usual Sobolev space with $s\in \R.$
	 
	\end{exmp}
	
	From Theorem \ref{result1} we see that $L^1$ integrability condition on lower order terms is sufficient to ensure that they do not influence the well-posedness and loss of regularity of the solution. The following example shows that the condition is not necessary.
	\begin{exmp}
			\begin{equation}
			\begin{cases}
				\begin{aligned}
					&\left( \partial_t^2 - \partial_x^2 - \frac{2}{t} \partial_x  \right)u(t,x) =0\\
					&u(0,x) =0, \; \partial_tu(0,x) = u_0(x).
				\end{aligned}
			\end{cases}
		\end{equation}
		This corresponds to the case $p=q=0$ and $r=1$ in (\ref{sing}). The solution to the above Cauchy problem is given by
		\[
		u(t,x) = tu_0(x+t).
		\]
	\end{exmp}
	
	The following example demonstrates that when the coefficient of the top order term is oscillatory but in $C^1((0,T]) \cap W^{1,1}((0,T])$ and that of the lower order term is in $ C^1((0,T]) \cap L^1((0,T])$, one may have no loss.
	\begin{exmp}
		\begin{equation}
			\begin{cases}
				\begin{aligned}
					&\left( \partial_t^2 - \big(2+ \sin \sqrt t \big)^2\partial_x^2 - \frac{\cos \sqrt t}{2\sqrt t} \partial_x\right)u(t,x) =0\\
					&u(0,x) = u_0(x), \; \partial_tu(0,x) = 2\partial_xu_0(x).
				\end{aligned}
			\end{cases}
		\end{equation}
		This corresponds to the case $p=0,q=\frac{1}{2}$ and $r=\frac{1}{2}$ in (\ref{sing}). The solution to the above Cauchy problem is given by
		\[
		u(t,x) = u_0\left( x + \int_{0}^{t} (2+ \sin \sqrt s)ds \right).
		\]
	\end{exmp}
	
	The following example demonstrates that one may encounter nonuniqueness when the lower order terms are not $L^1$ integrable.
	\begin{exmp}
		\begin{equation}
			\begin{cases}
				\begin{aligned}
					&\left( \partial_t^2 - \partial_x^2 - \frac{1}{t} \left( \partial_t + 3\partial_x\right)  \right)u(t,x) =0\\
					&u(0,x) =0, \; \partial_tu(0,x) = 0.
				\end{aligned}
			\end{cases}
		\end{equation}
		This corresponds to the case $p=q=0$ and $r=1$ in (\ref{sing}). The solution to the above Cauchy problem is given by
		\[
		u(t,x) = t^2 u_0(x+t),
		\]
		for any function $u_0(x).$
	\end{exmp} 
	\begin{rmk}
	  Observe that the form of solutions to the Cauchy problems in all the above examples suggests that the Cauchy data only propagates along one characteristic i.e., they are all one-way waves \cite{BR}.
	\end{rmk}

	\addcontentsline{toc}{section}{Acknowledgements}
	\section*{Acknowledgements}
	The first author is funded by the University Grants Commission, Government of India, under its JRF and SRF schemes.

\end{document}